\documentclass[12pt,twoside]{amsart}
\usepackage{amssymb}
\usepackage{amscd}
\usepackage{amsmath}
\usepackage[all]{xy}

\title[Effective freeness]
{Koll\'ar-type effective freeness for quasi-log canonical 
pairs}
\author{Osamu Fujino}
\date{2016/11/22, version 0.05}
\address{Department of Mathematics, Graduate School of Science,
Osaka University, Toyonaka, Osaka 560-0043, Japan}
\email{fujino@math.sci.osaka-u.ac.jp}

\newcommand{\Supp}[0]{\operatorname{Supp}}
\newcommand{\Nqlc}[0]{\operatorname{Nqlc}}
\newcommand{\Bs}[0]{\operatorname{Bs}}
\newtheorem{thm}{Theorem}[section]
\newtheorem{lem}[thm]{Lemma}

\newtheorem{cor}[thm]{Corollary}
\newtheorem*{claim}{Claim}

\theoremstyle{definition}

\newtheorem{rem}[thm]{Remark}
\newtheorem*{ack}{Acknowledgments}
\newtheorem{say}[thm]{}
\newtheorem{step}{Step}

\makeatletter
    
    \@addtoreset{equation}{section}
\makeatother

\begin{document}
\bibliographystyle{amsalpha+}

\maketitle

\begin{abstract}
We prove Koll\'ar-type effective basepoint-free 
theorems for quasi-log canonical pairs. 
\end{abstract}

\tableofcontents
\section{Introduction}\label{f-sec1}

In \cite{kollar}, J\'anos Koll\'ar established an effective 
basepoint-free theorem for kawamata log terminal pairs 
inspired by Demailly's paper \cite{demailly}. 
Of course, Demailly's arguments in \cite{demailly} 
are analytic. 
In \cite{kollar}, Koll\'ar wrote: 
\begin{quote}
The algebraic method works for $X$ singular and also for certain 
non ample line bundles. 
I formulated the most general case; readers interested in smooth varieties 
should just always set $\Delta=0$. 
\end{quote}

The author generalized Koll\'ar's 
effective basepoint-free theorem for log canonical pairs in 
\cite{fujino-effective1}. We note that 
the category of log canonical pairs is much wider than that of 
kawamata log terminal pairs and is the largest class for which 
the minimal model program can work (see \cite{fujino-fundamental}). 
Although the arguments in \cite{fujino-effective1} 
are algebraic and essentially the same as 
Koll\'ar's in \cite{kollar}, the vanishing theorems 
used in \cite{fujino-effective1} are much 
sharper than the Kawamata--Viehweg vanishing theorem, 
which is the main ingredient of \cite{kollar}. 
Since our vanishing theorems (see \cite{fujino-vanishing}, \cite{fujino-injectivity}, 
and \cite[Chapter 5]{fujino-foundation}) are sufficiently powerful, 
we can apply Koll\'ar's method in \cite{kollar} 
to quasi-log canonical pairs. 
We note that the notion of quasi-log canonical pairs was introduced by 
Florin Ambro in \cite{ambro}. 
We also note that 
any (quasi-projective semi-)log canonical pair can be 
seen as a quasi-log canonical pair (see \cite{fujino-slc}, 
\ref{f-say1.6} and Theorem \ref{f-thm1.7}). 

The main result of this paper is the following effective 
freeness and very ampleness for quasi-log canonical pairs. 

\begin{thm}[Effective freeness and very ampleness I]\label{f-thm1.1}
Let $[X, \omega]$ be a quasi-log canonical 
pair and let $\pi:X\to Y$ be a projective 
morphism between schemes. 
Let $D$ be a $\pi$-nef Cartier divisor on $X$ such that 
$aD-\omega$ is $\pi$-ample for some $a\geq 0$. 
Then there exists a positive integer $m=m(\dim X, a)$, which 
only depends on $\dim X$ and $a$, such that 
$\mathcal O_X(mD)$ is $\pi$-generated. 
We further assume that $D$ is $\pi$-ample. 
Then there exists a positive integer $m'=m'(\dim X, a)$ depending 
only on $\dim X$ and $a$ such that $\mathcal O_X(m'D)$ is $\pi$-very 
ample. 
\end{thm}

In Theorem \ref{f-thm1.1}, it is sufficient to 
assume that 
$aD-\omega$ is nef and log big over $Y$ with respect to 
$[X, \omega]$ when $\omega$ is a $\mathbb Q$-Cartier divisor 
(or a $\mathbb Q$-line bundle). 
Precisely speaking, we have: 

\begin{thm}[Effective freeness and very ampleness II]\label{f-thm1.2}
Let $[X, \omega]$ be a quasi-log canonical 
pair and let $\pi:X\to Y$ be a projective 
morphism between schemes. 
Let $D$ be a $\pi$-nef Cartier divisor on $X$ such that 
$aD-\omega$ is nef and log big over $Y$ with 
respect to $[X, \omega]$ for some $a\geq 0$. 
We further assume that $\omega$ is a $\mathbb Q$-Cartier 
divisor {\em{(}}or a $\mathbb Q$-line bundle{\em{)}} on $X$.  
Then there exists a positive integer $m=m(\dim X, a)$, which 
only depends on $\dim X$ and $a$, such that 
$\mathcal O_X(mD)$ is $\pi$-generated. 
We further assume that $D$ is $\pi$-ample. 
Then there exists a positive integer $m'=m'(\dim X, a)$ depending 
only on $\dim X$ and $a$ such that $\mathcal O_X(m'D)$ is $\pi$-very 
ample. 
\end{thm}

For the constants $m(\dim X, a)$ and $m'(\dim X, a)$ in 
Theorem \ref{f-thm1.1} and Theorem \ref{f-thm1.2}, 
we have: 

\begin{rem}\label{f-rem1.3}
We can take 
\begin{equation*}
m(\dim X, a)=2^{\dim X+1}(\dim X+1)! (\lceil a\rceil +\dim X)
\end{equation*}
and 
\begin{equation*}
m'(\dim X, a)=2^{\dim X+1}(\dim X+1)! (\lceil a \rceil +\dim X)
(\dim X+1)
\end{equation*} 
in Theorem \ref{f-thm1.1} and Theorem \ref{f-thm1.2}. 
The above constants are far from the optimal constants. 
\end{rem}

As a corollary of the above theorems, we have: 

\begin{cor}[Effective freeness and very ampleness 
for semi-log canonical 
pairs]\label{f-cor1.4}
Let $(X, \Delta)$ be a semi-log canonical 
pair and 
let $\pi:X\to Y$ be a projective morphism 
between schemes. 
Let $D$ be a $\pi$-nef Cartier divisor on $X$ and 
let $a$ be a nonnegative real number. 
We assume one of the following conditions: 
\begin{itemize}
\item[(1)] $aD-(K_X+\Delta)$ is ample over $Y$, or 
\item[(2)] $aD-(K_X+\Delta)$ is nef and log big over 
$Y$ with respect to $(X, \Delta)$ and $\Delta$ is a $\mathbb Q$-divisor. 
\end{itemize}
Then there exists a positive integer $m=m(\dim X, a)$, 
which only depends on $\dim X$ and $a$, such that 
$\mathcal O_X(mD)$ is $\pi$-generated. 
We further assume that $D$ is $\pi$-ample. 
Then there exists a positive integer $m'=m'(\dim X, a)$ depending 
only on $\dim X$ and $a$ such that $\mathcal O_X(m'D)$ is $\pi$-very 
ample. 
\end{cor}

Corollary \ref{f-cor1.4} is a slight generalization of 
\cite[Theorem 6.3]{fujino-slc}. 

\begin{rem}\label{f-rem1.5}
If $(X, \Delta)$ is kawamata log terminal in Corollary 
\ref{f-cor1.4}, 
then it is sufficient to assume that $\pi$ is proper, 
$aD-(K_X+\Delta)$ is nef and big over $Y$, 
and $\Delta$ is a $\mathbb Q$-divisor (see \cite{kollar}). 
For some comments, see Remark \ref{f-rem3.7} below.
\end{rem}

Let us quickly see some basic examples of quasi-log canonical 
pairs for the reader's convenience. 

\begin{say}[Quasi-log canonical pairs]\label{f-say1.6}
Let $(X, \Delta)$ be a log canonical pair and 
let $f:Y\to X$ be a resolution such that 
$K_Y+\Delta_Y=f^*(K_X+\Delta)$ and $\Supp \Delta_Y$ is a 
simple normal crossing divisor on $Y$. 
We put $B=\Delta^{<1}_Y$ and 
$\Delta_Y=S+B$. We consider 
the following short exact sequence 
$$
0\to \mathcal O_Y(-S+\lceil -B\rceil)\to 
\mathcal O_Y(\lceil -B\rceil)\to \mathcal O_S(\lceil -B\rceil)\to 0. 
$$ 
By the Kawamata--Viehweg vanishing theorem, we 
have $R^1f_*\mathcal O_Y(-S
+\lceil -B\rceil)=0$. 
Therefore, we obtain the 
short exact sequence  
$$
0\to \mathcal J(X, \Delta)\to \mathcal O_X\to f_*\mathcal O_S(\lceil -B\rceil)\to 0, 
$$ 
where $\mathcal J(X, \Delta)=f_*\mathcal O_Y(-S+\lceil -B\rceil)$ is the multiplier ideal 
sheaf of $(X, \Delta)$. 
Let $\mathrm{Nklt}(X, \Delta)$ denote the non-klt locus 
of $(X, \Delta)$ with the reduced scheme structure. 
Then we have $f_*\mathcal O_S(\lceil -B\rceil)\simeq 
\mathcal O_{\mathrm{Nklt}(X, \Delta)}$. 
This data 
$$
f:(S, B|_S)\to \mathrm{Nklt}(X, \Delta)
$$ 
is a typical example of quasi-log canonical pairs. 
We note that 
the data 
$$
f:(X, \Delta)\to X
$$ is also a quasi-log canonical pair since 
the natural map $\mathcal O_X\to f_*\mathcal O_Y(\lceil -B\rceil)$ is an 
isomorphism. Moreover, we have: 

\begin{thm}[{\cite[Theorem 1.2]{fujino-slc}}]\label{f-thm1.7}
Let $(X, \Delta)$ be a quasi-projective semi-log canonical 
pair. Then $[X, K_X+\Delta]$ is naturally a quasi-log canonical pair. 
\end{thm} 
Therefore, we can treat log canonical pairs, non-klt loci of log canonical 
pairs, and (quasi-projective) semi-log canonical pairs, and so on, in 
the same framework of quasi-log canonical pairs. 
For the precise statement and various applications of 
Theorem \ref{f-thm1.7}, see \cite{fujino-slc}. 
\end{say}

The following observation shows one of the advantages of 
using the framework of quasi-log canonical pairs. 

\begin{say}\label{f-say1.8}
Let $X'$ be any irreducible component of $X$ in 
Theorem \ref{f-thm1.1} or Theorem \ref{f-thm1.2}. 
Then, by the vanishing theorem for quasi-log canonical pairs 
(see, for example, \cite[Theorem 3.8]{fujino-reid-fukuda} and 
\cite[Chapter 6]{fujino-foundation}), 
$R^1\pi_*(\mathcal O_X(kD)\otimes \mathcal I_{X'})=0$ 
for every $k\geq a$, where 
$\mathcal I_{X'}$ is the defining ideal sheaf of $X'$. 
Therefore, the natural restriction map 
\begin{equation*}
\pi_*\mathcal O_X(kD)\to \pi_*\mathcal O_{X'}(kD)
\end{equation*} 
is surjective for every $k\geq a$. 
Moreover, by adjunction 
(see, for example, \cite[Theorem 3.8]{fujino-reid-fukuda} 
and \cite[Chapter 6]{fujino-foundation}), 
$[X', \omega|_{X'}]$ is quasi-log 
canonical. 
Thus, we may assume that $X$ is 
irreducible when we prove that 
$\mathcal O_X(mD)$ is $\pi$-generated 
for some $m\geq a$ in Theorem \ref{f-thm1.1} and Theorem \ref{f-thm1.2}. 
It is obvious that we may also assume that $Y$ is affine for the 
proof of Theorem \ref{f-thm1.1} and Theorem \ref{f-thm1.2}. 
\end{say}

Finally, we make some comments on the Angehrn--Siu-type effective 
freeness and point separation. 

\begin{say}\label{f-say1.9}
The Angehrn--Siu-type effective freeness and point separation 
for smooth projective varieties was first obtained in \cite{angehrn-siu}. 
Then it was generalized for kawamata log terminal pairs in 
\cite{kollar2} and for log canonical pairs in \cite{fujino-effective}. 
Recently, Haidong Liu established the Angehrn--Siu-type 
effective freeness and point separation for quasi-log canonical 
pairs in \cite{liu-san}. 
\end{say}

This paper is organized as follows. 
In Section \ref{f-sec2}, 
we prepare some useful lemmas for the proof of 
Theorem \ref{f-thm1.1} and Theorem \ref{f-thm1.2}. 
In Section \ref{f-sec3}, 
which is the main part of this paper, 
we translate Koll\'ar's method in \cite{kollar} 
into the framework of quasi-log canonical 
pairs. 
In Section \ref{f-sec4}, 
we prove an easy lemma on effective very ampleness. 
Section \ref{f-sec5} is devoted to the proof of Theorem 
\ref{f-thm1.1}, Theorem \ref{f-thm1.2}, and 
Corollary \ref{f-cor1.4}. 

\begin{ack}
The author was partially supported by 
JSPS KAKENHI Grant Numbers JP24684002, 
JP16H06337, JP16H03925. 
\end{ack}

We will work over $\mathbb C$, the complex number field, 
throughout this paper. For the details of the 
theory of quasi-log schemes, see \cite[Chapter 6]{fujino-foundation}. 
For a gentle introduction to the theory of quasi-log canonical pairs, 
we recommend the reader to see \cite{fujino-intro}. 
For the standard notations and conventions of the minimal 
model program, see \cite{fujino-fundamental}. 
In this paper, a {\em{scheme}} means a separated 
scheme of finite type over $\mathbb C$ and 
a {\em{variety}} means a reduced scheme. 

\section{Preliminaries}\label{f-sec2} 

We will freely use the theory of quasi-log 
schemes (see \cite[Chapter 6]{fujino-foundation}). 
We note that our formulation of quasi-log schemes 
is slightly different from Ambro's original one in \cite{ambro} 
(for the details, see 
\cite[Chapter 6]{fujino-foundation} and \cite{fujino-pull}). 

\begin{say}\label{f-say2.1}
Let $D=\sum _i a_i D_i$ be an $\mathbb R$-divisor, where 
$a_i\in \mathbb R$ and $D_i$ is a prime divisor for every $i$ such that 
$D_i\ne D_j$ for $i\ne j$. 
We put 
\begin{equation}
D^{<1}=\sum _{a_i<1} a_iD_i \quad \text{and}\quad 
\lceil D\rceil =\sum _i \lceil a_i \rceil D_i, 
\end{equation} 
where for every real number $x$, $\lceil x\rceil$ is the integer 
defined by $x\leq \lceil x\rceil <x+1$. 
We also put $\lfloor D\rfloor =-\lceil -D\rceil$ and $\{D \} =D-\lfloor D\rfloor$. 

Let $D$ be a Cartier divisor on $X$ and let $\pi:X\to Y$ be a proper 
morphism. 
Then $\Bs_\pi |D|$ denotes the relative base locus (with 
the reduced scheme structure) with respect to $\pi:X\to Y$. 
\end{say}

\begin{say}[Quasi-log canonical pairs]\label{f-say2.2}
Let $M$ be a smooth variety, let $Z$ be a simple normal crossing 
divisor on $M$, 
and let $B$ be an $\mathbb R$-divisor on $M$ such that 
$B$ and $Z$ have no common irreducible components and 
that $\Supp (B+Z)$ is a simple normal crossing divisor on $M$. 
Then $(Z, B|_Z)$ is called a globally embedded simple normal crossing pair. 

Let $X$ be a scheme and let $\omega$ be an $\mathbb R$-Cartier 
divisor (or an $\mathbb R$-line bundle) on $X$. 
Let $f:(Z, \Delta_Z)\to X$ be a proper morphism 
from a globally embedded simple normal crossing pair $(Z, \Delta_Z)$. 
If the natural map $\mathcal O_X\to 
f_*\mathcal O_Z(\lceil -(\Delta^{<1}_Z)\rceil)$ is an isomorphism 
and $f^*\omega\sim _{\mathbb R} K_Z+\Delta_Z$, then 
$[X, \omega]$ is called a quasi-log canonical 
pair (qlc pair, for short). 
We note that $f$ is not always birational and that $X$ is not necessarily 
equidimensional. 
For the details of quasi-log schemes, 
see \cite[Chapter 6]{fujino-foundation}. 
\end{say}

\begin{rem}\label{f-rem2.3}
For the details of (globally embedded) simple normal 
crossing pairs, 
see, for example, \cite[Chapter 5]{fujino-foundation}. 
For the definitions of quasi-log schemes, 
qlc strata, qlc centers, non-qlc loci, nef and 
log big divisors, and so on, 
see, for example, \cite[\S 3]{fujino-reid-fukuda} 
and \cite[Chapter 6]{fujino-foundation}. 
\end{rem}

Let us prepare some lemmas for the proof of the main results 
in the subsequent sections. 

\begin{lem}[{see \cite[Lemma 2.8]{liu-san}}]\label{f-lem2.4} 
Let $[X, \omega]$ be an irreducible 
qlc pair and let $B$ be an effective 
$\mathbb R$-Cartier divisor on $X$ such that 
$\Supp B$ contains no qlc centers of $[X, \omega]$. 
Then we can construct a natural quasi-log structure 
on $[X, \omega+B]$. 
\end{lem}

For the details of the quasi-log structure on $[X, \omega+B]$, 
see the construction in the proof below. 

\begin{proof}
We take a quasi-log resolution $f:(Z, \Delta_Z)\to X$, 
where $(Z, \Delta_Z)$ is a globally embedded simple normal crossing 
pair. 
By taking some blow-ups, 
we may assume that $\Supp f^*B\cup \Supp \Delta_Z$ is a simple 
normal crossing divisor on $Z$ since 
$\Supp B$ contains no qlc centers of $[X, \omega]$. 
We put $\Theta_Z=\Delta_Z+f^*B$. 
Then we have 
\begin{equation}
f_*\mathcal O_Z(\lceil -(\Theta^{<1}_Z)\rceil-\lfloor 
\Theta^{>1}_Z\rfloor)\subset 
f_*\mathcal O_Z(\lceil -(\Delta^{<1}_Z)\rceil)\simeq \mathcal O_X. 
\end{equation}
Therefore, by setting 
\begin{equation}
\mathcal I_{\Nqlc(X, \omega+B)}=f_*\mathcal O_Z(\lceil 
-(\Theta^{<1}_Z)\rceil-\lfloor 
\Theta^{>1}_Z\rfloor), 
\end{equation}
we have a quasi-log structure on $[X, \omega+B]$. 
By construction, this quasi-log structure coincides with the 
original quasi-log structure of $[X, \omega]$ outside $\Supp B$. 
\end{proof}

The following lemma is essentially the same as 
\cite[Lemma 3.16]{fujino-reid-fukuda}, which 
played a crucial role in the proof of 
the rationality theorem for quasi-log schemes 
(see \cite[Chapter 6]{fujino-foundation}). 

\begin{lem}\label{f-lem2.5}
Let $[X, \omega]$ be an irreducible qlc pair and let $V$ be an irreducible 
closed subvariety of $X$ such that 
$\mathrm{codim}_XV=k>0$. 
Let $D_0, D_1, \cdots, D_k$ be effective Cartier divisors on $X$ 
such that $V\subset \Supp D_i$ for every $i$. 
Assume that $\Supp D_i$ contains no qlc centers of $[X, \omega]$ for 
every $i$. 
Then $[X, \omega+\frac{1}{l}D_0+\sum _{i=1}^k D_i]$ is not 
qlc at the generic point of $V$ for every $l>0$. 
\end{lem}

In Lemma \ref{f-lem2.5}, we know that 
$[X, \omega+\frac{1}{l}D_0+\sum _{i=1}^kD_i]$ 
has a quasi-log structure, which 
coincides with the original quasi-log 
structure of $[X, \omega]$ outside 
$\sum_{i=0}^k \Supp D_i$, by Lemma \ref{f-lem2.4} since 
$\Supp D_i$ contains no qlc centers of $[X, \omega]$ 
for every $i$. 

\begin{proof}
We will get a contradiction by assuming that $[X, \omega+\frac{1}{l}D_0
+\sum _{i=1}^k D_i]$ is qlc at the generic point of $V$. 
By shrinking $X$, we may assume that $X$ is quasi-projective and 
that $[X, \omega+\frac{1}{l}D_0
+\sum _{i=1}^k D_i]$ is qlc everywhere. 
By taking general hyperplane cuts by adjunction 
(see, for example, \cite[Theorem 3.8]{fujino-reid-fukuda} 
and \cite[Chapter 6]{fujino-foundation}), 
we may assume that $\dim V=0$. 
By shrinking $X$ again, 
we may further assume that 
$V$ is a closed point of $X$. 
Let $f:(Z, \Delta_Z)\to X$ be a quasi-log resolution 
of $[X, \omega]$ 
such that $\Supp \Delta_Z\cup \sum _{i=0}^k \Supp f^*D_i$ is 
a simple normal crossing divisor on $Z$ 
(see the proof of Lemma \ref{f-lem2.4}). 
Let $X'$ be an irreducible component of 
$\Supp D_1$ passing through 
$V$. 
\begin{claim}There exists a stratum of $(Z, \Delta_Z+f^*D_1)$ which is 
mapped onto $X'$. 
\end{claim}
\begin{proof}[Proof of Claim]
Since $[X, \omega+D_1]$ is qlc, 
the coefficients of $\Delta_Z+f^*D_1$ are $\leq 1$. 
Assume that no strata of $(Z, \Delta_Z+f^*D_1)$ are mapped 
onto $X'$. Then we can easily check that 
$f^*D_1\leq \lceil -(\Delta^{<1}_Z)\rceil$ over the generic point of $X'$. 
Thus we have 
\begin{equation}
f_*\mathcal O_Z(\lceil -(\Delta^{<1}_Z)\rceil)\supseteq 
\mathcal O_X(D_1)\supsetneq \mathcal O_X
\end{equation} 
at the generic point of $X'$. 
On the other hand, $f_*\mathcal O_Z(\lceil -(\Delta^{<1}_Z)\rceil)\simeq 
\mathcal O_X$ always holds because $[X, \omega]$ is qlc. 
Therefore, we get a contradiction. 
This means that there exists a stratum of $(Z, \Delta_Z+f^*D_1)$ which is 
mapped onto $X'$. 
\end{proof}
Thus, $\Supp D_i$ does not contain $X'$ for $i\ne 1$ and $X'$ is 
a qlc center of 
$$\left[X, \omega+\frac{1}{l}D_0+\sum _{i=1}^k D_i\right] 
=\left[X, \omega+D_1+\frac{1}{l}D_0+\sum _{i=2}^k D_i\right]. $$ 
By adjunction, $[X', (\omega+D_1)|_{X'} +
\frac{1}{l} D_0|_{X'}+\sum _{i=2}^k D_k|_{X'}]$ is qlc. 
By repeating this argument $(k-1)$-times, 
we can reduce the problem to the case when 
$k=1$, that is, $\dim X=1$. 
In this case, we can easily get a contradiction. 
This is because there exists a stratum of 
$[Z, \Delta_Z+f^*D_1]$ which maps to $V$ as above. 
Therefore, $[X, \omega+\frac{1}{l}D_0+\sum _{i=1}^kD_i]$ is 
not qlc at the generic point of $V$ for every $l>0$. 
\end{proof}

\section{Koll\'ar's method for quasi-log canonical pairs}\label{f-sec3}

This section is a direct generalization of 
\cite[Section 2]{fujino-effective1} for quasi-log 
canonical pairs. We will translate Koll\'ar's 
method in \cite{kollar} into the framework of 
quasi-log canonical 
pairs. 

\begin{say}\label{f-say3.1}
Let $[X, \omega]$ be a qlc pair and let $\pi:X\to Y$ be a projective 
morphism 
onto an affine scheme $Y$. 
We assume that $X$ is irreducible and that $\omega$ is a 
$\mathbb Q$-Cartier divisor (or a $\mathbb Q$-line bundle) on $X$. 
Let $p: X\to Z$ be a proper 
surjective 
morphism onto a variety $Z$ over $Y$ with $p_*\mathcal O_X\simeq 
\mathcal O_Z$. 
We have the following commutative diagram. 
\begin{equation}
\xymatrix{
X\ar[r]^{p}\ar[d]_{\pi} & Z\ar[ld]^q\\
Y &
}
\end{equation}
Let $M$ be a $\pi$-semiample 
$\mathbb Q$-Cartier divisor on $X$ and 
let $N$ be a Cartier divisor on $X$. 
Assume that 
\begin{equation}
N\sim_{\mathbb R} \omega+B+M, 
\end{equation}
where $B$ is an effective $\mathbb Q$-Cartier 
divisor on $X$ such that $\Supp B$ contains no qlc centers of $[X, \omega]$ and that $B=p^*B_{Z}$, where 
$B_Z$ is an effective $q$-ample $\mathbb Q$-divisor on $Z$. 
We can construct a natural quasi-log structure on $[X, \omega+B]$ 
by Lemma \ref{f-lem2.4}. 
Let $X\setminus \Sigma$ be the largest Zariski open subset such that 
$[X, \omega+B]$ is qlc. 
Assume that $\Sigma\ne \emptyset$. 
Let $V$ be an irreducible component of $\Sigma$ such that 
$\dim p(V)$ is maximal.  
We note that $V_Z=p(V)$ is not equal to $Z$ since $B=p^*B_Z$. 
Let $f:(W, B_W)\to X$ be a quasi-log resolution of $[X, \omega]$. 
We can write 
\begin{equation}
K_W+B_W\sim _{\mathbb R} f^*\omega. 
\end{equation}
We have the following commutative diagram. 
\begin{equation}
\xymatrix{
W \ar[d]_f\ar[rd]^g &\\ 
X \ar[r]_p\ar[d]_\pi& Z\ar[ld]^q \\
Y 
}
\end{equation}
Without loss of generality, we may assume that 
\begin{equation}
\left(W, \Supp B_W\cup\Supp f^*B\right)
\end{equation} 
is a globally 
embedded simple normal crossing pair. 
We may further assume that 
$\Supp g^{-1} (V_Z)$ is a simple normal crossing divisor 
on $W$ and that 
\begin{equation}
\left(W, \Supp B_W\cup\Supp f^*B\cup\Supp g^{-1}(V_Z)\right)
\end{equation} 
is a globally 
embedded simple normal crossing pair. 
Let $c$ be the largest real number such that 
$[X, \omega+cB]$ is qlc over the generic point of $V_Z=p(V)$. 
We note that 
\begin{equation}
K_W+B_W+cf^*B\sim _{\mathbb R} f^*(\omega+cB). 
\end{equation} 
By the assumptions, we know that $0<c<1$. 
We can write 
\begin{equation}
f^*N\sim _{\mathbb R} K_W+B_W+cf^*B+(1-c) f^*B 
+f^*M. 
\end{equation}
We put 
\begin{equation}
\lfloor B_W+cf^*B\rfloor =F+G_1+G_2-H, 
\end{equation} 
where $F, G_1, G_2$, and $H$ are effective 
and have no common irreducible components such that 
\begin{itemize}
\item[(i)] the $g$-image of every irreducible component of $F$ is $V_Z=p(V)$, 
\item[(ii)] the $g$-image of every irreducible component of $G_1$ 
does not contain $V_Z$, 
\item[(iii)] the $g$-image of every irreducible component of 
$G_2$ contains $V_Z$ but does not coincide with 
$V_Z$, and 
\item[(iv)] $f_*\mathcal O_W(H)\simeq \mathcal O_X$ since 
$f_*\mathcal O_W(\lceil -(B^{<1}_W)\rceil)\simeq 
\mathcal O_X$ and $0\leq H\leq \lceil -(B^{<1}_W)\rceil$. 
\end{itemize}
By taking some more blow-ups, 
if necessary, we may assume that $F, G_1, G_2$, and 
$H$ are Cartier. 
Note that $G_2=\lfloor G_2\rfloor$ is a reduced 
simple normal crossing divisor on $W$ and that 
no stratum $C$ of $(W, G_2)$ satisfies $g(C)\subset V_Z$ 
since $\Supp g^{-1}(V_Z)$ is a simple normal crossing 
divisor on $W$. 
We also note that $F$ is reduced by construction. 
We put $N'=f^*N+H-G_1$. 
Let us consider the following short exact sequence: 
\begin{equation} 
0\to \mathcal O_W(N'-F)\to \mathcal O_W(N')\to 
\mathcal O_F(N')\to 0. 
\end{equation} 
We note that 
\begin{equation}
N'-F\sim _{\mathbb R} K_W+f^*M+(1-c)f^*B 
+\{B_W+cf^*B\} +G_2. 
\end{equation} 
Therefore, the connecting homomorphism 
\begin{equation}
g_*\mathcal O_F(N')\to R^1g_*\mathcal O_W(N'-F)
\end{equation} 
is a zero map (see, for example, \cite[Theorem 3.8]{fujino-reid-fukuda} 
and \cite[Chapter 6]{fujino-foundation}). 
Thus we obtain that 
\begin{equation}
0\to g_*\mathcal O_W(N'-F)\to g_*\mathcal O_W(N')
\to g_*\mathcal O_F(N')\to 0 
\end{equation} 
is exact. 
By the vanishing theorem (see, for example, \cite[Theorem 3.8]{fujino-reid-fukuda} 
and \cite[Chapter 6]{fujino-foundation}), 
we have 
\begin{equation}
H^i(Z, g_*\mathcal O_W(N'-F))=H^i(Z, g_*\mathcal O_W(N'))=0
\end{equation} 
for every $i>0$. 
Therefore, 
\begin{equation}\label{eq3.15}
H^0(Z, g_*\mathcal O_W(N'))\to 
H^0(Z, 
g_*\mathcal O_F(N'))
\end{equation} 
is surjective and $H^i(Z, g_*\mathcal O_F(N'))=0$ 
for every $i>0$. 
Thus we obtain 
\begin{equation}
\dim H^0(F_{\eta}, \mathcal O_{F_\eta}(N'))=
\chi ((V_Z)_{\eta}, g_*\mathcal O_F(N')|_{(V_Z)_{\eta}})
\end{equation}
where $\eta$ is the generic point of $\pi\circ f(F)=q\circ p(V)=q(V_Z)$. 
\end{say}

\begin{say}\label{f-say3.2}
In our application, $M$ will be a variable divisor of the form 
\begin{equation}
M_j=M_0+jL^0, 
\end{equation} 
where $M_0$ is a $\pi$-semiample 
$\mathbb Q$-divisor and $L^0=p^*L^0_Z$ with a 
$q$-ample Cartier divisor $L^0_Z$ on $Z$. 
Then we obtain that 
\begin{equation}
\begin{split}
&\dim H^0(F_\eta, \mathcal O_{F_\eta} (N'_0+jf^*L^0))
\\&=\chi ((V_Z)_\eta, (g_*\mathcal O_F(N'_0)
\otimes \mathcal O_Z(jL^0_Z))|_{(V_Z)_\eta})
\end{split}
\end{equation} 
is a polynomial in $j$ for $j\geq 0$, where 
\begin{equation}
N'_0=f^*N_0+H-G_1
\end{equation} 
and 
\begin{equation}
N_0\sim_{\mathbb R} \omega+B +M_0. 
\end{equation} 
\end{say}
\begin{say}\label{f-say3.3}
We assume that we establish 
$\dim H^0(F_\eta, \mathcal O_{F_\eta}(N'))\ne 0$. 
Then, by the above surjectivity \eqref{eq3.15}, 
we can lift sections to 
$H^0(Z, g_*\mathcal O_W(N'))\simeq 
H^0(W, \mathcal O_W(f^*N+H-G_1))$. 
Since $F\not\subset \Supp G_1$, we get a section 
$s\in H^0(W, \mathcal O_W(f^*N+H))$ which is 
not identically zero along $F$. 
We know that 
\begin{equation}
H^0(W, \mathcal O_W(f^*N+H))\simeq 
H^0(X, \mathcal O_X(N))
\end{equation} 
since $f_*\mathcal O_W(H)\simeq \mathcal O_X$. 
Therefore, $s$ descends to a section of $\mathcal O_X(N)$ which 
does not vanish along $f(F)$. 
\end{say}

\begin{lem}\label{f-lem3.4}
Let $[X, \omega]$ be a qlc pair such that 
$\omega$ is a $\mathbb Q$-Cartier divisor 
$($or a $\mathbb Q$-line bundle$)$ on $X$ and that 
$X$ is irreducible. 
Let $:p:X\to Z$ be a proper surjective morphism 
onto a variety $Z$ over an affine scheme $Y$ with 
$p_*\mathcal O_X\simeq \mathcal O_Z$. 
\begin{equation}
\xymatrix{
X\ar[r]^p\ar[d]_\pi& Z \ar[ld]^q\\ 
Y &
}
\end{equation}
Let $L^0_Z$ be a $q$-ample Cartier divisor on $Z$ and let 
$L_Z\sim 
mL^0_Z$ for some $m>0$. 
We put $L^0=p^*L^0_Z$ and 
$L=p^*L_Z$. 
Assume that $aL^0-\omega$ is nef and log big over 
$Y$ with 
respect to $[X, \omega]$ for some nonnegative real number 
$a$. 
Assume that 
$q_*\mathcal O_Z(L_Z)\neq 0$ and that $\Bs _\pi |L|$ contains 
no qlc centers of $[X, \omega]$, and 
let $V_Z\subset \Bs _q |L_Z|$ be an irreducible component with 
minimal $k=\mathrm{codim}_Z V_Z$. 
Then, with at most $\dim V_Z$ exceptions, $V_Z$ is not contained 
in $\Bs_q|kL_Z+(j+\lceil 2a\rceil +1)L^0_Z|$ for 
$j\geq 0$. 
\end{lem}

\begin{proof}
Let $V$ be an irreducible component of $p^{-1}(V_Z)$ 
such that $\mathrm{codim}_XV\leq k$ and $p(V)=V_Z$. 
Let $B_i$ be a general member of $|L|$ for 
every $0\leq i\leq k$. We put 
\begin{equation}
B=\frac{1}{2m}B_0+B_1+\cdots +B_k. 
\end{equation} 
Then $B\sim _{\mathbb Q} \frac{1}{2}L^0+kL$, 
$[X, \omega+B]$ is qlc outside $\Bs_\pi |L|$ and 
$[X, \omega+B]$ is not qlc at the 
generic point of $V$ by Lemma \ref{f-lem2.5}. 
We will apply the method in \ref{f-say3.1}, \ref{f-say3.2}, and \ref{f-say3.3} 
with 
\begin{equation}
N_j=kL+(j+\lceil 2a\rceil +1)L^0
\end{equation}
\begin{equation}
M_0 =\lceil 2a\rceil L^0-\omega +\frac{1}{2}L^0, 
\end{equation}
and 
\begin{equation}
M_j=M_0+jL^0. 
\end{equation}
We note that $M_j$ is $\pi$-semiample 
for every $j\geq 0$ by 
the basepoint-free theorem of Reid--Fukuda 
type for qlc pairs (see \cite[Chapter 6]{fujino-foundation} 
and \cite{fujino-reid-fukuda}) since $M_j$ is $\pi$-nef and 
$M_j-\omega$ is nef and log big over $Y$ with respect 
to $[X, \omega]$. 
The crucial point is to show that 
\begin{equation}
\dim H^0(F_\eta, \mathcal O_{F_\eta}(N'_j))
=\chi ((V_Z)_\eta, g_*\mathcal O_F(N'_j)|_{(V_Z)_\eta})
\end{equation} 
is not identically zero, where 
\begin{equation}
N'_j=f^*N_j +H-G_1
\end{equation} 
for every $j$. 
Let $C\subset F$ be a general fiber 
of $F\to g(F)=V_Z$. 
Then 
\begin{equation}
N'_0=(g^*(kL_Z+(\lceil 2a\rceil+1)L^0_Z)
+H-G_1)|_C=H|_C. 
\end{equation} 
Hence $g_*\mathcal O_F(N'_0)$ is not the zero sheaf, and 
\begin{equation}
H^0(F, \mathcal O_F(N'_j))=
H^0(Z, g_*\mathcal O_F(N'_0)\otimes 
\mathcal O_Z(jL^0_Z))\ne 0
\end{equation}
for $j\gg 1$. 
Therefore, $\dim H^0(F_\eta, \mathcal O_{F_\eta}(N'_j))$ is a nonzero 
polynomial of degree at most $\dim V_Z$ in $j$ for $j\geq 0$. 
Thus it can vanish for at most $\dim V_Z$ different 
values of $j$. 
This implies that 
\begin{equation}
\begin{split}
f(F)&\not\subset \Bs_\pi|kL+(j+\lceil 2a\rceil +1)L^0|
\\&=p^{-1}\Bs_q|kL_Z+(j+\lceil 2a\rceil +1)L^0_Z|
\end{split}
\end{equation}
with at most $\dim V_Z$ exceptions. 
Therefore, 
\begin{equation}
V_Z=g(F)\not\subset 
\Bs _q|kL_Z+(j+\lceil 2a\rceil +1)L^0_Z|
\end{equation} 
with at most $\dim V_Z$ exceptions. 
This is what we wanted. 
\end{proof}

\begin{cor}\label{f-cor3.5}
We use the same notation as in Lemma \ref{f-lem3.4}. 
We further assume that $m\geq 2a+\dim Z$. 
We put $k=\mathrm{codim}_Z\Bs_q|L_Z|$. 
Then we have 
\begin{equation}
\dim \Bs _q |(2k+2)L_Z|<\dim \Bs_q |L_Z|. 
\end{equation}
\end{cor}

\begin{proof}
It is obvious that $\Bs_q|(2k+2)L_Z|\subset \Bs_q |L_Z|$. 
Let $V_Z$ be a maximal dimensional 
irreducible component of $\Bs_q |L_Z|$. 
Then there is a value $0\leq j <\dim Z$ such that 
$V_Z$ is not in the base loci of 
\begin{equation}
|kL_Z+(j+\lceil 2a\rceil +1)L^0_Z| \quad \text{and} \quad 
|kL_Z+(2m-j-\lceil 2a\rceil -1)L^0_Z|
\end{equation} 
by Lemma \ref{f-lem3.4}. 
Therefore, we see that $V_Z$ is not in the base locus of 
\begin{equation}
\begin{split}
&|kL_Z+(j+\lceil 2a\rceil +1) L^0_Z +kL_Z
+(2m-j-\lceil 2a \rceil -1)L^0_Z|\\ &=|2kL_Z+2mL^0_Z|=|(2k+2)L_Z|. 
\end{split} 
\end{equation}
Thus, we obtain 
$\Bs_q|(2k+2)L_Z|\subsetneq \Bs_q |L_Z|$. 
\end{proof}

\begin{lem}\label{f-lem3.6}
Let $\pi:X\to Y$ be a projective morphism from an irreducible 
qlc pair $[X, \omega]$ to an affine scheme $Y$. 
Let $D$ be a $\pi$-nef Cartier divisor on $X$ such that 
$aD-\omega$ is nef and log big over $Y$ with 
respect to $[X, \omega]$. 
Then we can find an effective Cartier divisor $D'\in |2(\lceil a\rceil +\dim X)D|$ 
such that $\Supp D'$ contains no qlc centers of $[X, \omega]$. 
\end{lem}

\begin{proof}
Let $S$ be an arbitrary qlc stratum of $[X, \omega]$. 
Let us consider the following short exact sequence: 
\begin{equation}
0\to \mathcal I_S\otimes \mathcal O_X(jD)\to \mathcal O_X(jD)\to 
\mathcal O_S(jD)\to 0
\end{equation}
where $\mathcal I_S$ is the defining ideal sheaf of $S$. 
By the vanishing theorem for qlc pairs 
(see, for example, \cite[Theorem 3.8]{fujino-reid-fukuda} and 
\cite[Chapter 6]{fujino-foundation}), 
\begin{equation}\label{eq3.37}
R^i\pi_*(\mathcal I_S\otimes \mathcal O_X(jD))=
R^i\pi_*\mathcal O_X(jD)=R^i\pi_*\mathcal O_S(jD)=0 
\end{equation} 
for every $i\geq 1$ and $j\geq a$. 
Let $\eta$ be the generic point of $\pi(S)$. 
Then we have that 
\begin{equation}
\dim H^0(S_\eta, \mathcal O_{S_\eta} (jD))
=\chi (S_\eta, \mathcal O_{S_\eta}(jD))
\end{equation} 
holds for $j\geq a$. 
We note that $\chi(S_\eta, \mathcal O_{S_\eta}(jD))$ is a nonzero 
polynomial in $j$ because $|mD|$ is basepoint-free for $m\gg 0$. 
We note that the natural restriction map 
\begin{equation}
\pi_*\mathcal O_X(jD)\to \pi_*\mathcal O_S(jD)
\end{equation} 
is surjective for every $j\geq a$ by \eqref{eq3.37}. 
Therefore, with at most $\dim S_\eta$ exceptions, 
$S\not\subset \Bs_\pi|(\lceil a\rceil +j)D|$ for $j\geq 0$. 
Thus, we can take an effective Cartier divisor $D'\in |2(\lceil a\rceil +\dim X)D|$ 
such that $D'$ contains no qlc centers of $[X, \omega]$ 
by the same argument as in the proof of 
Corollary \ref{f-cor3.5}. 
\end{proof}

\begin{rem}\label{f-rem3.7}
In this section, we need the projectivity of $\pi:X\to Y$ 
only when we use the basepoint-free theorem of Reid--Fukuda 
type for qlc pairs (see \cite[Chapter 6]{fujino-foundation} and 
\cite{fujino-reid-fukuda}). 
Therefore, if we apply the method explained in this section to 
kawamata log terminal pairs, then it is sufficient to 
assume that $\pi:X\to Y$ is only proper by the 
usual Kawamata--Shokurov basepoint-free theorem. 
\end{rem}

\section{Effective very ampleness lemma}\label{f-sec4}

The following easy lemma is a quasi-log canonical 
version of \cite[Lemma 7.1]{fujino-fujita}. 

\begin{lem}[Effective very ampleness lemma]\label{f-lem4.1}
Let $[X, \omega]$ be a quasi-log canonical 
pair and let $\pi:X\to Y$ be 
a projective morphism between schemes. 
Let $A$ be a $\pi$-ample Cartier divisor on $X$ 
such that $\mathcal O_X(A)$ is $\pi$-generated. 
Assume that $A-\omega$ is 
nef and log big over $Y$ with respect to $[X, \omega]$. 
Then $(\dim X+1)A$ is very ample over $Y$. 
\end{lem}

\begin{proof}By the vanishing theorem for qlc pairs 
(see, for example, \cite[Theorem 3.8]{fujino-reid-fukuda} 
and \cite[Chapter 6]{fujino-foundation}), 
we have 
\begin{equation}
R^i\pi_*\mathcal O_X((n+1-i)A)=0
\end{equation} for 
every $i>0$, where $n=\dim X$. 
Then, by the Castelnuovo--Mumford regularity, 
we obtain that 
\begin{equation}
\pi_*\mathcal O_X(A)\otimes \pi_*\mathcal O_X(mA)
\to \pi_*\mathcal O_X((m+1)A)
\end{equation} 
is surjective for every $m\geq n+1$ 
(see, for example, \cite{kleiman} and \cite[Example 1.8.24]{lazarsfeld}). 
Therefore, we see that 
\begin{equation}
\mathrm{Sym}^k\pi_*\mathcal O_X((n+1)A)\to \pi_*\mathcal 
O_X(k(n+1)A)
\end{equation} 
is surjective for every $k\geq 1$. 
This implies that $(n+1)A$ is $\pi$-very ample. 
\end{proof}

We note that \cite[1.2 Lemma]{kollar} does not seem to be true 
as stated. 
Therefore, we will use Lemma \ref{f-lem4.1} 
in the proof of Theorem \ref{f-thm1.1} and Theorem \ref{f-thm1.2} in 
Section \ref{f-sec5}. 

\section{Proof of Main Theorems}\label{f-sec5}

In this final section, we prove Theorem \ref{f-thm1.1} and 
Theorem \ref{f-thm1.2}. 

\begin{proof}[Proof of Theorems \ref{f-thm1.1} and \ref{f-thm1.2}]
In Step \ref{f-step1}, we will prove the effective basepoint-free 
theorems in Theorem \ref{f-thm1.1} and Theorem \ref{f-thm1.2}. 
Then we will check the effective very ampleness in Step \ref{f-step2}. 
\begin{step}[Effective freeness]\label{f-step1}
By shrinking $Y$, we may assume that $Y$ is affine. 
By perturbing $\omega$ when $aD-\omega$ is $\pi$-ample, 
we may further assume that 
$\omega$ is a $\mathbb Q$-Cartier divisor (or a $\mathbb Q$-line 
bundle) on $X$. 
By the observation in \ref{f-say1.8}, we may assume that 
$X$ is irreducible. 
By Lemma \ref{f-lem3.6}, we can find an effective 
Cartier divisor $D'\in |2(\lceil a \rceil +\dim X)D|$ such that 
$\Supp D'$ contains no qlc centers of $[X, \omega]$. 
By the basepoint-free theorem for qlc pairs (see, for example, 
\cite[Chapter 6]{fujino-foundation}), 
we have the following commutative diagram: 
\begin{equation}
\xymatrix{
X\ar[r]^{p}\ar[d]_{\pi} & Z\ar[ld]^q\\
Y &
}
\end{equation}
such that $p_*\mathcal O_X\simeq 
\mathcal O_Z$ and that $D\sim p^*D_Z$ for some 
$q$-ample Cartier divisor $D_Z$ on $Z$. 
By applying Lemma \ref{f-lem3.4} repeatedly 
in order to lower the dimension of 
$\dim \Bs_q |mD_Z|$, we finally obtain that 
\begin{equation}
|2^{\dim X+1}(\dim X+1)! (\lceil a\rceil +\dim X)D|
\end{equation} is 
basepoint-free. 
\end{step}
\begin{step}[Effective very ampleness]\label{f-step2}
By combining the effective basepoint-free theorem in Step \ref{f-step1} with 
Lemma \ref{f-lem4.1}, we obtain that 
\begin{equation}
2^{\dim X+1}(\dim X+1)! (\lceil a\rceil +\dim X)(\dim X+1)D
\end{equation} is $\pi$-very ample. 
\end{step}
We have completed the proof of Theorem \ref{f-thm1.1} and 
Theorem \ref{f-thm1.2}. 
\end{proof}

We close this section with the proof of Corollary \ref{f-cor1.4}. 

\begin{proof}[Proof of Corollary \ref{f-cor1.4}]
Without loss of generality, we may assume that 
$Y$ is affine. 
Then $X$ is quasi-projective since $\pi$ is projective. 
By \cite[Theorem 1.2]{fujino-slc} (see Theorem \ref{f-thm1.7}), 
$[X, K_X+\Delta]$ has a natural quasi-log structure compatible 
with the original semi-log canonical structure of $(X, \Delta)$. 
For the details, see \cite{fujino-slc}. 
Then, by applying Theorem \ref{f-thm1.1} and Theorem \ref{f-thm1.2} 
to $[X, K_X+\Delta]$, we obtain the desired statements. 
\end{proof}
%%%%%%%%%%%%%%%%%%%%%%%%%%%

\end{document}